\documentclass[12pt]{article}

\usepackage{amsfonts}
\usepackage{amssymb}
\usepackage{amsthm}
\usepackage{amsmath}
\usepackage{mathrsfs}
\usepackage[latin1]{inputenc}
\usepackage[all]{xy}
\usepackage[left=1.20in,right=1.20in]{geometry}
\usepackage{xcolor}

\usepackage{listings}

\def\uno{\mathbf{1}}

\theoremstyle{plain}
\newtheorem{teo}{Theorem}[section]
\newtheorem{prop}[teo]{Proposition}
\newtheorem{coro}[teo]{Corollary}
\newtheorem{lema}[teo]{Lemma}

\theoremstyle{definition}
\newtheorem{defi}[teo]{Definition}
\newtheorem{nota}[teo]{Notation}
\newtheorem{rem}[teo]{Remark}

\theoremstyle{remark}

\title{On the essential algebra  of the shifted Burnside biset functor}
\author{Nadia Romero\footnote{\texttt{nadia.romero@ugto.mx}}\\ 
\begin{small}
Departamento de Matem\'aticas,
\end{small}\\
\begin{small}
Universidad de Guanajuato, Mexico.
\end{small}
\date{ }
}

\begin{document}

\maketitle

\begin{abstract}
We describe the essential algebra, $\widehat{kB_T}(G)$, of the Burnside biset functor shifted by a group $T$, at a group $G$, in two cases. First, when $G$ and $T$ are both finite abelian groups and $k$ is a field of characteristic $0$. In this case, $\widehat{kB_T}(G)$ is isomorphic to a quotient of the \textit{shifted star algebra}, which is defined in terms of the subgroups of $G\times G\times T$. The second case is  when $G$ and $T$ are any finite groups satisfying $(|G|,\, |T|)=1$ and $k$ is a commutative unitary ring. In this case, $\widehat{kB_T}(G)$ is isomorphic to a semidirect product of $Out(G)$ and $kB^{Z(G)}(T)$, the monomial Burnside ring of $T$ with coefficients in $Z(G)$. 

The aim of the article is to consider the natural set of generators of $\widehat{kB_T}(G)$ coming from the transitive elements in $kB_T(G\times G)$ and explore some cases in which it is possible to give a basis for $\widehat{kB_T}(G)$ in this set. 

$ $\\
\noindent{\textbf{Keywords}: Essential algebra, shifted Burnside biset functor, star operation, monomial Burnside ring.}

\noindent{\textbf{AMS MSC (2010)}: 16S99, 16D25, 18B99, 19A22.}
\end{abstract}

\section*{Introduction}

The notion of an essential algebra in the context of biset functors appeared first in the  classification of simple  biset functors (see, for example,  \cite{ensemble} and \cite{biset}). 
Latter on, it showed to be a key element in the problem of classifying simple modules over an arbitrary Green biset functor (see  \cite{maxime} and \cite{lachica2}). More recently, the notion appeared in other related context, such as the correspondence functors (\cite{corr1} and \cite{corr2}) and the  diagonal $p$-permutation functors (\cite{b-deniz}).  On the other hand, the shifted biset functors are now known to be a fundamental part of the theory of biset functors (see, for example, \cite{centros}).

The article is divided in  four sections.  In Section 1 we recall  basic results regarding the shifted Burnside biset functor, $kB_T$, for $k$ a commutative ring with unity and $T$ a finite group. 
 For further details on the theory of biset functors, please see \cite{biset}.  
 In Section 2 we generalize some results by Boltje and Danz (\cite{boltje-danz}) regarding the star operation, obtaining \textit{shifted} versions of them.
In particular, if $k$ is a field of characteristic $0$, if $G$ is a finite group and if $\mathscr{S}_{GGT}$ denotes the set of subgroups of $G\times G\times T$, we show that there exists an injective ring homomorphism, non-unitary in general,
\begin{displaymath}
\alpha_{GGT}:k B_T(G\times G)\rightarrow k\mathscr{S}_{GGT},
\end{displaymath}
where, of course, $k\mathscr{S}_{GGT}$ has a suitable star algebra structure. Furthermore, at the end of Section 2 we define an essential algebra $\widehat{k\mathscr{S}}_{GGT}$ for $k\mathscr{S}_{GGT}$ and see that $\alpha_{GGT}$ extends to a (in general non-unitary) ring morphism $\widehat{\alpha}_{GGT}:\widehat{kB_T}(G)\rightarrow\widehat{k\mathscr{S}}_{GGT}$. In Theorem \ref{iso2abe} we show that if $G$ and $T$ are finite abelian groups, then  $\widehat{\alpha}_{GGT}$ is a ring isomorphism.

The main result of the article is Theorem \ref{elteo}.  Here we come back to the general case and show that for any groups $G$ and $T$ and $k$ a commutative unitary ring, if we consider the monomial Burnside ring $kB^{Z(G)}(T)$, then  there is an injective algebra homomorphism
\begin{displaymath}
\Phi: kB^{Z(G)}(T)\rtimes Out(G)\rightarrow kB_T(G\times G),
\end{displaymath}
which actually extends injectively to $\widehat{kB_T}(G)$. As a consequence, in Corollary \ref{isoprimos}, we show that if  $G$ and $T$ are groups such that $(|G|,\, |T|)=1$, then   $\widehat{kB_T}(G)$ is isomorphic to $kB^{Z(G)}(T)\rtimes Out(G)$. Furthermore, in this case we have an isomorphism of $k$-modules,
\begin{displaymath}
kB_T(G\times G)\cong (kB^{Z(G)}(T)\rtimes Out(G))\oplus \mathcal{I}_T(G),
\end{displaymath}
where $\mathcal{I}_T(G)$ is the ideal of $kB_T(G\times G)$ generated by elements  that factor through groups $K$ such that $|K|<|G|$. This generalizes the results of Bouc for the case $T=1$ (see Section 4 in \cite{ensemble}). 

Along Section 3, we deal with some subsets of the set of generators of $\widehat{kB_T}(G)$ coming from the transitive elements in $kB_T(G\times G)$ that appear naturally in this context. So, in Section 4 we give examples, with $k$ being a field of characteristic $0$, to show that these sets may be properly contained in one another. The objective of this is to show that explicitly giving a basis for $\widehat{kB_T}(G)$ in this set of generators, in the general case, may not be an easy task. We also show that the morphism $\widehat{\alpha}_{GGT}$ is not injective in general.

A bit of the chronology of the article is worth mentioning. 	The first result to be conjectured by the author was a primitive version of Corollary \ref{isoprimos}. Then, Serge Bouc kindly provided wonderful help to confection a GAP  program (\cite{GAP4}) to test this conjecture, which was rapidly confirmed in several examples. After observing other examples generated by the program at the same time, Theorem \ref{iso2abe} was conjectured too. At this stage, the GAP program had became more sophisticated, the last version of it appears in subsection 4.3. 


\section{Preliminaries}

We fix the following notation.

Throughout the article, $G$, $H$ and $K$ will be finite groups. The trivial group will be denoted by $\uno$ and a cyclic group of order $n>1$ by $C_n$.

Whenever we have three or more finite groups, we will write $GHK$ for $G\times H\times K$.

If $a,\,g\in G$, we write $^ag$ for $aga^{-1}$ and $g^a$ for $a^{-1}ga$. Also, $c_a$ denotes the conjugation on $G$ by $a\in G$, sending $g\in G$ to $aga^{-1}$.

If $\sigma :G\rightarrow G$ is an automorphism of $G$, then $\Delta_{\sigma}(G)$ is the subgroup of $G\times G$, $\{(\sigma(g),\,g)\mid g\in G\}$. If $\sigma$ is the identity, then we simply write $\Delta(G)$.

We write $\mathscr{S}_G$ for the set of subgroups of $G$ and $[\mathscr{S}_G]$ for a set of representatives of the conjugacy classes of $\mathscr{S}_G$.

In this section, $k$ will be a commutative unitary ring and $T$ will be a finite group.  

\subsection{The shifted Burnside functor}

We recall some results from sections 2.3 and 8.2 of \cite{biset}, and Section 4.2 of \cite{lachica}. Lemma 2.3.25 in \cite{biset} will be called Goursat's lemma.


Recall that the biset functor $kB_T$ sends a finite group $G$ to the Burnside group $kB_T(G)= kB(G\times T)$. In a finite $(H,\, G)$-biset $X$, it defines a morphism of $k$-modules 
\begin{displaymath}
kB_T(X):kB_T(G)\rightarrow kB_T(H);\quad y\mapsto kB(X\times T)(y).
\end{displaymath}
It is straightforward to see that if $Y$ is a $(G\times T)$-set, then $kB(X\times T)(Y)$ is isomorphic, as $(H\times T)$-set, to $X\times_G Y$ with the action of $T$ given by $t[x,\, y]=[x,\, ty]$, for $[x,\,y]$ in $X\times_GY$. We will denote this operation by $X\times_G^dY$. The superscript $d$ stands for \textit{diagonal}, we will see next why this notation makes sense.

The category $\mathcal{P}_{kB_T}$, associated to $kB_T$, has for objects the class of finite groups and as set of arrows from a group $G$ to a group $H$, the $k$-module $kB_T(H\times G)$. Composition in $\mathcal{P}_{kB_T}$ is given in the following way: if $X$ is a $HGT$-set and $Y$ is a $GKT$-set, then $X\circ Y$ is the $HKT$-set $X\times_G^dY$, that is, it is the $H\times K$-set $X\times_GY$ with diagonal action of $T$, $t[x,\, y]=[tx,\, ty]$. The operation of the previous paragraph can be seen as a particular case of this composition, with $K=\uno$ and the action of $T$ on $X$ being trivial.

For a group $G$, we have then the endomorphism algebra $kB_T(G\times G)$, with identity given by the isomorphism class of $(G G T)/(\Delta(G) T)$.

\begin{nota}
\label{notastar}
Let $D$ be a subgroup of $HKT$. We will write $p_1(D)$, $p_2(D)$ and $p_3(D)$ for the
projections of $D$ in $H$, $K$ and $T$ respectively, $p_{1,2}(D)$ will denote the projection over $H\times K$, and in the same way we define the other possible combinations of subscripts. We write $k_1(D)$ for $\{h\in H\mid (h,\,1,\,1)\in D\}$, which is a normal subgroup of $p_1(D)$. Similarly, we define $k_2(D)$, $k_3(D)$ and $k_{i,j}(D)$ for all possible combinations of $i$ and $j$.

If $E$ is  a subgroup of $KGT$, we denote by $D\ast E$ the set
\begin{displaymath}
\{(h,\, g,\, t)\in HGT\mid \exists x\in K \textrm{ s.t. } (h,\, x,\,t)\in D \textrm{ and }(x,\, g,\, t)\in E\}.
\end{displaymath}
\end{nota}

We have the following results.

\begin{lema}
\label{stabs}
Let $G$, $H$ and $K$ be groups,  $U$ be an $HGT$-set and $V$ be a $KHT$-set. Then if $u\in U$ and $v\in V$, the stabilizer of $[v,\, u]\in V\times_H^dU$ under the action of $KGT$ is
\begin{displaymath}
Stab_{KGT}([v,\, u])=Stab_{KHT}(v)\ast Stab_{HGT}(u).
\end{displaymath}
\end{lema}
\begin{proof}
The proof is a straightforward generalization of that of Lemma 2.3.20 in \cite{biset}.
\end{proof}

The next two lemmas are proved in \cite{lachica}. We include the statements here, since it will be useful to have them at hand in the next sections. 

\begin{lema}[Lemma 4.5 in \cite{lachica}]
\label{mackey}
Let $G$, $L$, $K$ and $T$ be groups. If $E$ is a subgroup of $GLT$ and $D$ is a subgroup of $LKT$,
then we have an isomorphism 
\begin{displaymath}
(G L T)/E \times_L^d (L K T)/D \cong \bigsqcup_{(l,\, t)} (G K T)/(E \ast ^{(l,\, 1,\,t)}D) 
\end{displaymath}
of $GKT$-sets. Here $(l,\, t)$ runs through a set of representatives of the double cosets $p_{2,3}(E)\backslash L\times T / p_{1,3}(D)$.
\end{lema}

\begin{lema}[Lemma 4.8 in \cite{lachica}]
\label{alabouc}
If $D$ is a subgroup of $HKT$, then the $HKT$-set $(H K T)/D$ is isomorphic to
\begin{itemize}
\item[i)] $X\times_{D_1}^dY$, where $D_1=p_1(D)/k_1(D)$, for some $X\in kB_T(H \times D_1)$ and $Y\in kB_T(D_1\times K)$.
\item[ii)] $W\times_{D_2}^dZ$ where $D_2=p_2(D)/k_2(D)$, for some $W\in kB_T(H\times D_2)$ and $Z\in kB_T(D_2\times K)$.
\end{itemize}
\end{lema}

The last basic result of this section is a straightforward generalization of point 2 of Lemma 2.3.22 in \cite{biset}.

\begin{lema}
Let $G$, $H$ and $K$ be groups. Let $L$ be a subgroup of $HGT$ and $M$ be a subgroup of $KHT$. There are inclusions of subgroups
\begin{displaymath}
k_1(M)\subseteq k_1(M\ast L)\subseteq p_1(M\ast L)\subseteq p_1(M),
\end{displaymath} 
\begin{displaymath}
k_2(L)\subseteq k_2(M\ast L)\subseteq p_2(M\ast L)\subseteq p_2(L),
\end{displaymath}
\begin{displaymath}
k_3(M)\cap k_3(L)\subseteq k_3(M\ast L)\subseteq p_3(M\ast L)\subseteq p_3(M)\cap p_3(L).
\end{displaymath}
\end{lema}

\subsection{The essential algebra}

The essential algebra of $kB_T$ at a group $G$ is defined as the quotient of the endomorphism algebra of $G$ in $\mathcal{P}_{kB_T}$, that is $kB_T(G\times G)$ with product $\times_G^d$, over the ideal $\mathcal{I}_T(G)$, generated by elements of the form $\alpha\times_K^d\beta$, where $\alpha\in kB_T(G\times K)$, $\beta\in kB_T(K\times G)$ and $K$ is a group such that $|K|<|G|$. It is denoted by $\widehat{kB_T}(G)$.

\begin{defi}
\label{elmorf}
We denote by $E_{T,\, G}$ the natural ring epimorphism
\begin{displaymath}
E_{T,\, G}:kB_T(G\times G)\rightarrow \widehat{kB_T}(G).
\end{displaymath}
We will write $[x]$ for the image of an element $x\in kB_T(G\times G)$ under $E_{T,\, G}$.
\end{defi} 

 The classes of the transitive $GGT$-sets, denoted by $[(GGT)/D]$, form a $k$-basis for $kB_T(G\times G)$. Hence, $[[(GGT)/D]]$, their images under $E_{T,\, G}$,  form a generating set for $\widehat{kB_T}(G)$. 
The goal in Section 3 will be to find, in some cases, a $k$-basis of $\widehat{kB_T}(G)$ in this set.


\begin{lema}\label{comblin}
Suppose that the element $a_1[e_1]+\cdots +a_n[e_n]$, for some different transitive elements $[e_1],\ldots ,[e_n]$ in  $kB_T(G\times G)$, lies in the kernel of $E_{T,\, G}$. Then, if $a_j\neq 0$, for some $j$, and  $e_j=GGT/D$, we have that $D$ can be written as $D=E\ast F$, with $E\leqslant GKT$  and $F\leqslant KGT$ for some group $K$ such that $|K|<|G|$.
\end{lema}
\begin{proof}
Since
\begin{displaymath}
a_1[[e_1]]+\cdots +a_n[[e_n]]=0
\end{displaymath}
in $\widehat{kB_T}(G)$, we have
\begin{displaymath}
a_1[e_1]+\cdots +a_n[e_n]=\sum_{i=1}^rb_i(x_i\times_{K_i}^dy_i).
\end{displaymath}
for some $x_i\in kB_T(G\times K_i)$, $y_i\in kB_T(K_i\times G)$, $b_i\in k$ and $K_i$ groups such that $|K_i|<|G|$. If we write the elements in the right-hand side of the equation in terms of transitive $GGT$-sets, we obtain, by Lemma \ref{mackey}, that for some $i$ there exist  $A_i\leqslant GK_iT$ and $B_i\leqslant K_iGT$ such that
\begin{displaymath}
(GGT)/D\cong (GGT)/(A_i\ast B_i)
\end{displaymath}
as $GGT$-sets. Hence $D=(A_i\ast B_i)^{(g,\, g',\, t)}$ for some $(g,\, g',\, t)\in GGT$. But it is easy to see that the conjugate $(A_i\ast B_i)^{(g,\, g',\, t)}$ is equal to $A_i^{(g,\, 1,\, t)}\ast B_i^{(1,\, g',\, t)}$. So, taking $K=K_i$, $E=A_i^{(g,\, 1,\, t)}$ and $F=B_i^{(1,\, g',\, t)}$ we obtain the result.
\end{proof}


\subsection{The monomial Burnside ring}

The monomial Burnside ring will appear at the end of the paper, so we recall its definition and the description of its basis given by transitive sets (see \cite{dmon}). For further results about it see \cite{bamon}, \cite{lachica2} or \cite{bolco}. 


\begin{defi}
Let $C$ be a finite abelian group and $G$ be a finite group. A finite $C$-free $(G\times C)$-set is called a 
$C$-fibred $G$-set. The monomial Burnside ring for $G$ with coefficients in $C$, denoted by $kB^C(G)$, is, as a $k$-module, the submodule  of $kB(G\times C)$ generated by the the classes of the $C$-fibred $G$-sets. 
\end{defi}

The product in $kB^C(G)$ is given in the following way. Let $X$ and $Y$ be $C$-fibred $G$-sets and regard $X$ as having the action of $C$ on the right. Then consider $X\times_CY$ with the $(G\times C)$-action
\begin{displaymath}
(g,\, c)[x,\, y]=[(g,\, c)x,\, (g,\, 1)y]=[(g,\,1)x,\, (g,\, c)y],
\end{displaymath}
for $(g,\, c)\in G\times C$ and $[x,\, y]\in X\times Y$. It is clear that this action is well defined and that $C$-acts freely in $X\times_CY$. This product extends linearly to $kB^C(G)$ and is denoted simply by $uv$ for $u$ and $v$ in $kB^C(G)$. The identity element in $kB^C(G)$ is the class of the trivial $C$-fibred $G$-set $C$. 

\begin{nota}
\label{delmono}
If $D$ is a subgroup of $G$ and $\varphi:D\rightarrow C$ is a group homomorphism, we write $D_\varphi$ for the group $\{(x,\, \varphi(x))\mid x\in D\}$ and denote by $(D,\,\varphi)$ the transitive $C$-fibred $G$-set $(G\times C)/D_\varphi$.  We write $[D,\, \varphi]$ for the class of $(D,\,\varphi)$ in $kB^C(G)$. Notice that if $(g,\,x)\in G\times C$, then the conjugate  $^{(g,\, x)}D_{\varphi}$ is equal to $(^gD)_{\varphi c_{g^{-1}}}$, where $\varphi c_{g^{-1}}:{ }^gD\rightarrow C$. We write $^g\varphi $ for the morphism $\varphi c_{g^{-1}}$, so that $[D,\,\varphi]=[{ }^gD,\, { }^g\varphi]$. 
\end{nota}

\begin{rem}
\label{base}
Recall that the elements $[D,\,\varphi]$ for $D\leqslant G$ and $\varphi:D\rightarrow C$ a group homomorphism, form a basis of $kB^C(G)$. Moreover, if $[E,\, \psi]$ is another element of this type, their product in $kB^C(G)$ is given by
\begin{displaymath}
[D,\,\varphi][E,\,\psi]=\sum_{g\in[D\backslash G/E]}[D\cap { }^gE, \varphi\cdot { }^g\psi],
\end{displaymath}
where $[D\backslash G/E]$ is a set of representatives of $D\backslash G/E$ and the dot in the morphism $\varphi\cdot { }^g\psi:D\cap { }^gE\rightarrow C$ stands for the product in $C$.
\end{rem}

We finish this subsection by noticing that if $D\leqslant G$, then $N_G(D)$ acts on the right on the set of group homomorphisms, $Hom(D,\, C)$. The action is given by sending $g\in G$  and $f\in Hom(D,\, C)$ to $fc_g$, that is, $fc_g(x)=f({ }^gx)$ for $x\in D$. In this case,
\begin{displaymath}
[D,\, fc_g]=[{ }^gD,\, { }^g(fc_g)]=[D,\, f].
\end{displaymath}
In turn, this implies that the cardinality of the basis in Remark \ref{base} is equal to $\sum_{D\in [\mathscr{S}_G]}|\overline{Hom}(D,\, C)|$, where $\overline{Hom}(D,\,C)$ denotes the set of orbits of $Hom(D,\, C)$ under the action of $N_G(D)$.

\section{The star operation}

Throughout this section, $k$ is a field of characteristic $0$.

Following Boltje and Danz \cite{boltje-danz}, 
 we observe that for groups $G$, $H$ and $K$  we can define an operation
\begin{displaymath}
\mathscr{S}_{GHT}\times \mathscr{S}_{HKT}\longrightarrow \mathscr{S}_{GKT}
\end{displaymath}
sending a couple of subgroups $(D,\, E)$ to $D\ast E$. This operation is easily seen to be associative. Also, taking $\Delta(H)\times T\in \mathscr{S}_{HHT}$, gives $(\Delta(H)\times T)*L=L$ for any $L\in \mathscr{S}_{HHT}$. Hence, we can define a category $\mathcal{C}$, having as objects the class of all finite groups and as morphisms from $H$ to $G$, the set $\mathscr{S}_{GHT}$ (or the vector space $k\mathscr{S}_{GHT}$), with composition given by the star operation.

We denote by $\alpha_{GT}$ the morphism of $k$-vector spaces, $\alpha_{G\times T}$, defined in \cite{boltje-danz} as
\begin{displaymath}
\alpha_{GT}:k B_T(G)\rightarrow k\mathscr{S}_{G\times  T};\ [X]\mapsto \sum_{x\in X}Stab_{G\times T}(X),
\end{displaymath}
for a $G\times T$-set $X$. As it is pointed out in Proposition 3.2 of \cite{boltje-danz}, this map is injective and $\alpha_{GT}([G\times T/U])=[N_{G\times T}(U):U]\sum_{U'}U'$, where $U'$ runs over the conjugacy class of $U$. In order to make this morphism compatible with the composition in the category $\mathcal{P}_{k B_T}$, 
\begin{displaymath}
\times_H^d:k B_T(G\times H)\times k B_T(H\times K)\rightarrow k B_T(G\times K), 
\end{displaymath}
it is necessary to adjust the star composition $D\ast E$. Surprisingly, the adjustment made in the case considered in \cite{boltje-danz} works also in this case. That is, for $D$, a subgroup of  $GHT$, and $E$, a subgroup of $HKT$, consider
\begin{displaymath}
\kappa(D,\, E)=\frac{|k_2(D)\cap k_1(E)|}{|H|}
\end{displaymath}
in $k$, and define a $k$-bilinear map
\begin{displaymath}
k\mathscr{S}_{GHT}\times k\mathscr{S}_{HKT}\longrightarrow k\mathscr{S}_{GKT}
\end{displaymath}
by $D\ast_H^\kappa E=\kappa (D,\, E)D\ast E$. 

We have the following generalization of Proposition 3.5 in \cite{boltje-danz}.

\begin{lema}
Let $G$, $H$, $K$, and $I$ be finite groups. If $\kappa$ is defined as before, then for any $L\leqslant GHT$, $M\leqslant HKT$ and $N\leqslant KIT$, we have
\begin{displaymath}
\kappa(L,\, M)\kappa(L\ast M,\, N)=\kappa(L,\, M\ast N)\kappa(M,\,N).
\end{displaymath}
\end{lema}
\begin{proof}
To prove that we have the following group isomorphism
\begin{displaymath}
\frac{k_2(L)\cap k_1(M\ast N)}{k_2(L)\cap k_1(M)}\cong \frac{k_2(L\ast M)\cap k_1(N)}{k_2(M)\cap k_1(N)}
\end{displaymath}
we follow exactly the same steps of the proof of Proposition 3.5 in \cite{boltje-danz}.
\end{proof}

\begin{coro}
Let $D$ be a subgroup of  $GHT$, $E$ be a subgroup of $HKT$ and $F$ be a subgroup of $KLT$, then
\begin{displaymath}
(D\ast_H^\kappa E)\ast_K^\kappa F=D\ast_H^\kappa (E\ast_K^\kappa F).
\end{displaymath}
In particular, the $k$-vector space $k\mathscr{S}_{G\times G}$, together with multiplication $\ast_G^\kappa$ is a $k$-algebra with identity element $|G|(\Delta(G)\times T)$.
\end{coro}

Finally, we prove the analogue of Proposition 3.6 in \cite{boltje-danz}.

\begin{lema}
\label{abre}
Let $G$, $H$ and $K$ be finite groups. For $a$ in $k B_T(G\times H)$ and $b$ in $k B_T(H\times K)$ we have
\begin{displaymath}
\alpha_{GKT}(a\times_H^db)=\alpha_{GHT}(a)\ast_H^\kappa\alpha_{HKT}(b).
\end{displaymath}
\end{lema}
\begin{proof}
Suppose that $a=[X]$ and $b=[Y]$, where $X$ is a finite $GHT$-set and $Y$ is a finite $HKT$-set. Then
\begin{eqnarray*}
\alpha_{GKT}([X\times_H^dY]) & = & \sum_{[x,\, y]\in X\times_H^d Y} Stab_{GKT}([x,\, y])\\
& = & \sum_{[x,\, y]\in X\times_H^d Y}Stab_{GHT}(x)\ast Stab_{HKT}(y)\quad \textrm{by Lemma \ref{stabs}}.
\end{eqnarray*}
But this is equal to
\begin{displaymath}
\sum_{(x,\, y)\in X\times Y}\frac{|k_2(Stab_{GHT}(x))\cap k_1(Stab_{HKT}(y))|}{|H|} Stab_{GHT}(x)\ast Stab_{HKT}(y)
\end{displaymath}
since $X\times_H^dY$, as a set, is equal to $X\times_H Y$, so the number of elements in the orbit of $(x,\, y)\in X\times Y$ under the action of $H$ is $[H:k_2(Stab_{GHT}(x))\cap k_1(Stab_{HKT}(y))]$. The sum above is clearly equal to $\alpha_{GHT}(a)\ast_H^\kappa\alpha_{HKT}(b)$.
\end{proof}

\begin{rem}
\label{iso1abe}
Suppose $G$ and $T$ are abelian. Then $\alpha_{GGT}$ is surjective and hence it is an isomorphism of $k$-algebras between $(kB_T(G\times G),\, \times_G^d)$ and $(k\mathscr{S}_{GGT},\, \ast_G^\kappa)$.
\end{rem}

The following corollary  is the analogue of Lemma \ref{alabouc} for the star operation.

\begin{coro}
\label{alastar}
If $D$ is a subgroup of $HKT$, then 
\begin{itemize}
\item[i)] $D=X\ast Y$, for some $X\leqslant H D_1T$ and $Y\leqslant D_1 KT$, with $D_1=p_1(D)/k_1(D)$.
\item[ii)] $D=W\ast Z$, for some $W\leqslant HD_2T$ and $Z\leqslant D_2KT$, with $D_2=p_2(D)/k_2(D)$.
\end{itemize}
\end{coro}
\begin{proof}
The result follows from lemmas \ref{alabouc} and \ref{abre}, since the subgroups of $HKT$ form a basis for $k\mathscr{S}_{HKT}$ and $D$ appears in $\alpha_{HKT}([HKT/D])$ with non-zero coefficient.
\end{proof}

We will not go further on the generalization of \cite{boltje-danz}, the rest of the paper goes in a different direction. Nevertheless, the author believes that such generalization should be possible and that a deeper study of  the morphism $\alpha$ could lead to a complete description of the essential algebra for the shifted Burnside biset functor, since the star operation will be a key element in the descriptions we will give in the next section. In this sense, we begin by defining the essential algebra for the \textit{star algebra}.

\begin{defi}
\label{stares}
The essential algebra, $\widehat{k\mathscr{S}}_{GGT}$, for the algebra $(k\mathscr{S}_{GGT},\,\ast_G^{\kappa})$ is its quotient by the ideal $\mathcal{J}$ generated by elements of the form $\alpha \ast_K^\kappa\beta$, where $\alpha\in k\mathscr{S}_{GKT}$, $\beta\in k\mathscr{S}_{KGT}$ and $K$ is a group such that $|K|<|G|$.
\end{defi}


Notice that if $E\leqslant GKT$ and $F\leqslant KGT$ with $|K|<|G|$, then $E\ast F$ and $E\ast_K^{\kappa}F$ differ by a non-zero element in $k$. Hence,  $\mathcal{J}$ is equal to the ideal generated by elements of the form $\alpha \ast\beta$, where $\alpha\in k\mathscr{S}_{GKT}$, $\beta\in k\mathscr{S}_{KGT}$ and $K$ is a group such that $|K|<|G|$. 

We have a (non unitary) ring homomorphism, extending $\alpha_{GGT}$,
\begin{displaymath}
\widehat{\alpha}_{GGT}:\widehat{kB_T}(G)\rightarrow\widehat{k\mathscr{S}}_{GGT}.
\end{displaymath}
In Section 4 we will see that $\widehat{\alpha}_{GGT}$ is not injective in general. On the other hand, we will see in the next section that if $G$ and $T$ are abelian, then $\widehat{\alpha}_{GGT}$ is an isomorphism of $k$-algebras.


\section{Description in two cases}

Throughout this section $T$ is a finite group and, unless stated otherwise, $k$ is a commutative ring with unity.

The first observation is that, by Lemma \ref{alabouc}, we know that if $[[(GGT)/D]]$ is not zero in $\widehat{kB_T}(G)$, then $p_1(D)=G$, $k_1(D)=\uno$, $p_2(D)=G$ and $k_2(D)=\uno$. These conditions are equivalent to saying that $D$ can be written as
\begin{displaymath}
D=\{(f(g,\, t),\, g,\, t)\mid (g,\, t)\in A\},
\end{displaymath}
where $A\leqslant G\times T$ satisfies $p_1(A)=G$ and $f:A\rightarrow G$ is a surjective group homomorphism such that $Ker f\cap (k_1(A)\times \uno)=\uno$. 

\begin{defi}
\label{notzeroimplies}
The image under $E_{T,\, G}$ (see Definition \ref{elmorf}) of the set of transitive elements $[GGT/D]$ in $kB_T(G\times G)$ for which $D$ is of the form just described will be denoted by $\mathcal{G}$.
\end{defi}

Notice that there might be elements in $\mathcal{G}$ which are equal to zero, since Lemma \ref{alabouc} only gives a \textit{necessary} condition for $[[(GGT)/D]]$ not to be zero in $\widehat{kB_T}(G)$. Nevertheless, the set $\mathcal{G}$ is a generating set for $\widehat{kB_T}(G)$. 

The following lemma is an immediate consequence of Lemma \ref{comblin}.

\begin{lema}
\label{Einy}
Let $\mathcal{S}'$ be the set of elements $[GGT/D]$ in $kB_T(G\times G)$ for which $D$ cannot be written as $D=E\ast F$, with $E\leqslant GHT$, \mbox{$F\leqslant HGT$} and $H$ a group such that $|H|<|G|$. Then, no non-trivial linear combination of the elements of $\mathcal{S}'$ is in the kernel of $E_{T,\, G}$. In particular,  $E_{T,\, G}$ is injective in $\mathcal{S'}$ and $E_{T,\,G}(\mathcal{S'})$ is a linearly independent subset of $\widehat{kB_T}(G)$.
\end{lema}

If $\mathcal{S'}$ is as in the previous lemma, we have that for every $[GGT/D]$ in $\mathcal{S}'$,
\begin{displaymath}
E_{T,\, G}([GGT/D])=[[GGT/D]]\neq 0\textrm{ in } \widehat{kB_T}(G).
\end{displaymath}
Hence, by Lemma \ref{alabouc}, $E_{T,\,G}(\mathcal{S'})$ is contained in $\mathcal{G}$.

\begin{defi}
\label{defls}
We will denote by $\mathcal{L}_{s'}$ the set $E_{T,\,G}(\mathcal{S'})$. The $\mathcal{L}$ stands for \textit{linearly independent} and the $s'$ for \textit{not} (decomposable in terms of the) \textit{star composition}. 
\end{defi}

We will see next that if $G$ and $T$ are finite abelian groups and $k$ is a field of characteristic $0$, then $\mathcal{L}_{s'}$ is a basis of  $\widehat{kB_T}(G)$.



\begin{lema}
Suppose that $D$ is a subgroup of $GGT$ such that  $D$ can be written as $D=E\ast F$, with $E\leqslant GHT$ and $F\leqslant HGT$, for some group $H$. Then $D=E_1\ast F_1$, with $E_1\leqslant GH_1T$, $F_1\leqslant H_1GT$ and $H_1=p_2(E)/k_2(E)$.
\end{lema}
\begin{proof}
Let $K=k_2(E)$ and $H_0=p_2(E)$. Consider $E$ as a subgroup of $G H_0 T$ and define
\begin{displaymath}
F_0=\{(h,\, g,\,t)\in F\mid h\in  H_0\}.
\end{displaymath}
It is clear that $F_0\leqslant H_0GT$. Now we take 
\begin{displaymath}
\rho:G H_0 T\rightarrow \frac{G H_0 T}{\uno K \uno}\cong GH_1T\quad
\textrm{ and }\quad
\lambda:H_0 G T\rightarrow \frac{H_0 G T}{K \uno \uno}\cong H_1GT.
\end{displaymath}
Let $E_1$ be the image of $E$ under $\rho$ and $F_1$ be the image of $F_0$ under $\lambda$.

Let us show that $D=E_1\ast F_1$. Let $(g,\, g',\, t)\in D$. Since $D=E\ast F$, there exists $h\in H$ such that $(g,\, h,\, t)\in E$ and $(h,\, g',\, t)\in F$. Hence $h\in p_2(E)=H_0$ and $(h,\, g',\, t)\in F_0$. So $(g,\, hK,\, t)\in E_1$, $(hK,\, g',\, t)\in F_1$ and $(g,\, g',\,t)\in E_1\ast F_1$. Now suppose that $(a,\, a',\, t)\in E_1\ast F_1$. Then there exists $fK\in H_1$ such that $(a,\, fK,\, t)\in E_1$ and $(fK,\, a',\, t)\in F_1$. This means that there exist $h$ and $h'$ in $p_2(E)$ with $hK=fK=h'K$ and such that $(a,\, h,\, t)\in E$ and $(h',\, a',\, t)\in F_0$.  Now, since $h^{-1}h'\in K=k_2(E)$, we have $(1,\, h^{-1}h',\,1)\in E$ and so $(a,\, h',\, t)\in E$. Also, $(h',\, a',\, t)\in F$. So we have $(a,\, a',\,t)\in E\ast F=D$.
\end{proof}

Remark that the previous lemma also holds if we take 
$H_1=p_1(F)/k_1(F)$. 

\begin{teo}
\label{iso2abe}
Let $k$ be a field of characteristic $0$ and  $G$ and $T$ be finite abelian groups. Then $\mathcal{L}_{s'}$ is a basis of $\widehat{kB_T}(G)$ and we have an isomorphism of $k$-algebras
\begin{displaymath}
(\widehat{kB_T}(G), \times_G^d)\cong (\widehat{k\mathscr{S}}_{GGT}, \ast_G^\kappa).
\end{displaymath}
\end{teo}
\begin{proof}
Suppose that $D$ can be written as $D=E\ast F$, with $E\leqslant GHT$  and $F\leqslant HGT$ for some group $H$ such that $|H|<|G|$. By the previous lemma, $D=E_1\ast F_1$ with $E_1\leqslant GH_1T$  and $F_1\leqslant H_1GT$, where $H_1=p_2(E)/k_2(E)$. By Goursat's lemma, $H_1$ is isomorphic to a subquotient of $G\times T$, hence it is an abelian group. Also, since it is a subquotient of $H$, it satisfies $|H_1|<|G|$.



Consider the elements $[GH_1T/E_1]\in kB_T(G\times H_1)$ and $[H_1GT/F_1]\in kB_T(H_1\times G)$. Composing them gives
\begin{displaymath}
(GH_1T)/E_1 \times_{H_1}^d (H_1GT)/F_1 \cong \bigsqcup_{(h,\, t)} (GGT)/(E_1 \ast ^{(h,\, 1,\,t)}F_1) 
\end{displaymath}
where $(h,\, t)$ runs through a set of representatives of $p_{2,3}(E_1)\backslash H_1\times T / p_{1,3}(F_1)$. Since $H_1$ and $T$ are abelian, we have
\begin{displaymath}
[(GH_1T)/E_1] \times_{H_1}^d [(H_1GT)/F_1]=m[(GGT)/(E_1 \ast F_1)], 
\end{displaymath}
in $kB_T(G\times G)$, where $m=[H_1\times T:p_{2,3}(E_1)p_{1,3}(F_1)]$. 

This proves that $[[GGT/D]]$ is not zero in $\widehat{kB_T}(G)$ if and only if $[[GGT/D]]$ is in $\mathcal{L}_{s'}$, and hence that $\mathcal{L}_{s'}$ is a basis of $\widehat{kB_T}(G)$. It also proves that $\alpha_{GGT}(\mathcal{I}_T(G))=\mathcal{J}$, the ideal introduced in Definition \ref{stares}. Hence, by Remark \ref{iso1abe}, we have $(\widehat{kB_T}(G), \times_G^d)\cong (\widehat{k\mathscr{S}}_{GGT}, \ast_G^\kappa)$.
\end{proof}


\subsection{General setting}

Now we come back to the general case and suppose that $k$ is a commutative unitary ring. We will describe some elements in $\mathcal{L}_{s'}$.

Suppose that $[GGT/D]$ is as in Definition \ref{notzeroimplies} and that $A$ is of the form $A=G\times T_0$  for some $T_0\leqslant T$. Then $f(g,\, t)=f(g,\, 1)f(1,\, t)$ and the restriction of $f$ to $G\times\uno$ defines an automorphism  $\sigma$ of $G$. This is because $k_1(A)=G$ and, as said at the beginning of the section, $Ker f\cap (k_1(A)\times \uno)=\uno$. On the other hand, the restriction of $f$ to $T_0$ defines  a group homomorphism $\alpha:T_0\rightarrow Z(G)$. In this case, we write $D$ as:
\begin{displaymath}
D_{\sigma,\, \alpha,\, T_0}=\{(\alpha(t)\sigma(g),\, g,\, t)\mid (g,\, t)\in G\times T_0\}.
\end{displaymath}

\begin{defi}
\label{deflp}
The set of transitive elements $[[GGT/D_{\sigma,\, \alpha,\, T_0}]]$ in $\widehat{kB_T}(G)$ as above will be denoted by $\mathcal{L}_p$. The $p$ here stands for (decomposable as a) \textit{direct product}.
\end{defi}

\begin{prop}
\label{lpinls}
The set $\mathcal{L}_p$ is contained in $\mathcal{L}_{s'}$.
\end{prop}
\begin{proof}
A particular case of this proposition was proven in  Lemma 4.10 of \cite{lachica}. 

Let $\alpha$ and $\sigma$ be as before and suppose  $D_{\sigma,\, \alpha,\, T_0}=E\ast F$ for some $E\leqslant GHT$, $F\leqslant HGT$ and $H$ a group such that $|H|<|G|$. We claim that in this case $\Delta_{\sigma}(G)= k_{1,\, 2}(E)\ast k_{1,\, 2}(F)$. This will give a contradiction since, if we call $M$ the group on the right-hand side of this equality, then by Lemma 2.3.22 in \cite{biset}, $p_1(M)/k_1(M)$ is isomorphic to a subquotient of $H$.

Let $(\sigma(g),\, g)\in \Delta_{\sigma}(G)$.  Since $(\sigma(g),\, g,\, 1)$ is in $D_{\sigma,\, \alpha,\, T_0}$, there exists $h\in H$ such that $(\sigma(g),\, h,\, 1)\in E$ and $(h,\, g,\, 1)\in F$, thus 
$(\sigma(g),\, h)\in k_{1,\, 2}(E)$ and $(h,\, g)\in k_{1,\, 2}(F)$. Hence $(\sigma(g),\, g)\in k_{1,\, 2}(E)\ast k_{1,\, 2}(F)$. Now let $(g_1,\,g_2)\in k_{1,\, 2}(E)\ast k_{1,\, 2}(F)$. Then there exists $h\in H$ such that $(g_1,\, h)\in k_{1,\, 2}(E)$ and $(h,\, g_2)\in k_{1,\, 2}(F)$, thus $(g_1,\, h,\, 1)\in E$ and $(h,\, g_2,\, 1)\in F$. This means that $(g_1,\, g_2,\, 1)$ is in $D_{\sigma,\, \alpha,\, T_0}$. But given an element $(\sigma(g)\alpha(t),\, g,\, t)$ in $D_{\sigma,\, \alpha,\, T_0}$, the elements $g$ and $t$ are uniquely determined. This means that $g_1=\sigma(g_2)$ and $(g_1,\, g_2)\in \Delta_{\sigma}(G)$. 
\end{proof}

Next we prove some properties of the elements of $\mathcal{L}_p$.

\begin{lema}
\label{sonigual}
Let $D_{\sigma,\, \alpha,\, T_0}$ and $D_{\tau,\, \beta,\, T_1}$ be as before, with $\sigma$ and $\tau$ automorphisms of $G$ and  $\alpha: T_0\rightarrow Z(G)$ and $\beta: T_1\rightarrow Z(G)$ group homomorphisms. Then $[GGT/D_{\sigma,\, \alpha,\, T_0}]=[GGT/D_{\tau,\,\beta,\, T_1}]$ in $kB_T(G\times G)$ if and only if there exists $(g,\, t)\in G\times T$ such that $T_0= ^t\!\!T_1$, $\sigma=\tau c_g$ and $\alpha=\beta c_{t^{-1}}$.
\end{lema}
\begin{proof}
If $[GGT/D_{\sigma,\, \alpha,\, T_0}]=[GGT/D_{\tau,\,\beta,\, T_1}]$ in $kB_T(G\times G)$, then there exist $g_1$, $g_2$ in $G$ and $t$ in $T$ such that $D_{\sigma,\, \alpha,\, T_0}=\,^{(g_1,\, g_2,\, t)}D_{\tau,\, \beta,\, T_1}$. This easily implies that $T_0=\, ^tT_1$, that $\alpha (t')=\beta (t^{-1}t't)$ and that 
\begin{displaymath}
\sigma (g')=g_1\tau(g_2^{-1}g'g_2)g_1^{-1}=\tau(\tau^{-1}(g_1)g_2^{-1}g'g_2\tau^{-1}(g_1^{-1})),
\end{displaymath}
for $t'\in T_0$ and $g'\in G$. Taking $g=\tau^{-1}(g_1)g_2^{-1}$, we have $\sigma=\tau c_g$.

Now suppose that $T_0=\,^tT_1$ for some $t\in T$, $\sigma=\tau c_g$ for some $g\in G$ and
$\alpha=\beta c_{t^{-1}}$. Then, it is straightforward to see that
\begin{displaymath}
D_{\sigma,\, \alpha,\, T_0}=\,^{(\tau(g),\, 1,\, t)}D_{\tau,\, \beta,\, T_1}, 
\end{displaymath}
hence $[GGT/D_{\sigma,\, \alpha,\, T_0}]=[GGT/D_{\tau,\,\beta,\, T_1}]$ in $kB_T(G\times G)$.
\end{proof}




\begin{lema}
\label{composicion}
For $D_{\sigma,\, \alpha,\, T_0}$ and $D_{\tau,\, \beta,\, T_1}$ as before, we have
\begin{displaymath}
(GGT/D_{\sigma,\, \alpha,\, T_0})\times_G^d (GGT/D_{\tau,\, \beta,\, T_1})\cong \bigsqcup_{t\in[T_0\backslash T/T_1]}GGT/D_{\sigma\tau,\, \alpha\cdot(\sigma\,{^t\!\beta}),\,T_0\cap { }^tT_1},
\end{displaymath}
where, in the group morphism $\alpha\cdot(\sigma\,{ }^t\!\beta): T_0\cap { }^tT_1\rightarrow Z(G)$, the dot $\cdot$ represents the product in $Z(G)$. 
\end{lema}
\begin{proof}
By Lemma \ref{mackey} we have
\begin{displaymath}
(GGT/D_{\sigma,\, \alpha,\, T_0})\times_G^d (GGT/D_{\tau,\, \beta,\, T_1})\cong \bigsqcup_{(g,\, t)}GGT/(D_{\sigma,\, \alpha,\, T_0}\ast { }^{(g,\, 1,\, t)}D_{\tau,\, \beta,\, T_1})
\end{displaymath}
where $(g,\, t)$ runs through a set of representatives of $p_{2,\, 3}(D_{\sigma,\, \alpha,\, T_0})\backslash G\times T/p_{1,\, 3}(D_{\tau,\, \beta,\, T_1})$. Now, since $\alpha:T_0\rightarrow Z(G)$ and $\beta: T_1\rightarrow Z(G)$, then $p_{2,\, 3}(D_{\sigma,\, \alpha,\, T_0})=G\times T_0$ and $p_{1,\, 3}(D_{\tau,\, \beta,\, T_1})=G\times T_1$. So, $p_{2,\, 3}(D_{\sigma,\, \alpha,\, T_0})\backslash G\times T/p_{1,\, 3}(D_{\tau,\, \beta,\, T_1})$ is in one-to-one corres\-pondence with $T_0\backslash T/T_1$. Also, as in Lemma \ref{sonigual}, we see that ${ }^{(1,\, 1,\, t)}D_{\tau,\, \beta,\, T_1}=D_{\tau,\, { }^t\beta,\,{ }^tT_1}$. Hence,
\begin{displaymath}
(GGT/D_{\sigma,\, \alpha,\, T_0})\times_G^d (GGT/D_{\tau,\, \beta,\, T_1})\cong \bigsqcup_{t\in[T_0\backslash T/T_1]}GGT/(D_{\sigma,\, \alpha,\, T_0}\ast D_{\tau,\, { }^t\beta,\, { }^tT_1}).
\end{displaymath} 
In general, $D_{\sigma,\, \alpha,\, T_0}\ast D_{\tau,\, \gamma,\, T_2}$ is equal to
\begin{displaymath}
\{(g_1,\, g_2,\, t')\in GGT\mid \exists g_0\in G\textrm{ s. t. } (g_1,\, g_0,\, t')\in D_{\sigma,\, \alpha,\, T_0},\, (g_0,\, g_2,\, t')\in D_{\tau,\, \gamma,\, T_2}\}.
\end{displaymath}
This implies that $t'$ is in $T_0\cap T_2$, that $g_0=\tau(g_2)\gamma(t')$ and finally that  $g_1$ is equal to $\sigma\tau(g_2)\sigma\gamma(t')\alpha(t')$. That is,
\begin{displaymath}
D_{\sigma,\, \alpha,\, T_0}\ast D_{\tau,\, \gamma,\, T_2}=D_{\sigma\tau,\, \alpha\cdot(\sigma\gamma),\,T_0\cap T_2 },
\end{displaymath}
with $\alpha\cdot(\sigma\gamma): T_0\cap T_2\rightarrow Z(G)$. So, we have
\begin{displaymath}
(GGT/D_{\sigma,\, \alpha,\, T_0})\times_G^d (GGT/D_{\tau,\, \beta,\, T_1})\cong \bigsqcup_{t\in[T_0\backslash T/T_1]}GGT/D_{\sigma\tau,\, \alpha\cdot(\sigma\, { }^t\!\beta),\,T_0\cap { }^tT_1},
\end{displaymath}
where for each $t\in [T_0\backslash T/T_1]$, $\alpha\cdot(\sigma\, { }^t\!\beta): T_0\cap { }^tT_1\rightarrow Z(G)$.
\end{proof}



The previous lemma says that, given groups $G$ and $T$,  the $k$-submodule of $\widehat{kB_T}(G)$ generated by $\mathcal{L}_p$ is actually a subalgebra, the identity element of $\widehat{kB_T}(G)$ corresponds to $[GGT/D_{id,\,\uno,\, T}]$, where $id:G\rightarrow G$ is the identity morphism and $\uno :T\rightarrow Z(G)$ is the trivial morphism. We will 
use the  monomial Burnside ring, $kB^{Z(G)}(T)$, to describe this algebra. 

Let us fix the notation $Z$ for $Z(G)$. 

Let $U$ be a $Z$-fibred $T$-set. By considering $U$ as a left $Z$-set and $G$ as a right $Z$-set, we can make the composition $G\times_ZU$. 
Next, given $\sigma$, an automorphism of $G$, we turn $G\times_ZU$ into a $GGT$-set in the following way. Let $[g,\, u]$ be the class in $G\times_ZU$ of $(g, u)\in G\times U$ and let $(g_1,\,g_2,\,t)\in GGT$, then
\begin{displaymath}
(g_1,\,g_2,\,t)[g,\, u]=[g_1g\sigma(g_2)^{-1},\,(t,\,1)u].
\end{displaymath}
Since the elements of $Z$ commute with those of $T$ in $T\times Z$ and, of course, with those of $G$, this action is well defined. We denote the class of this $GGT$-set in $kB_T(G\times G)$ by $[U,\, \sigma]$ .

\begin{lema}
\label{corres}
Let $\sigma$ be an automorphism of $G$, let $X$ be a subgroup of $T$ and let $\alpha:X\rightarrow Z$ be a group homomorphism. Recall that $X_{\alpha}=\{(x,\, \alpha(x))\mid x\in X\}$ and that $(X,\, \alpha)$ denotes the $Z$-fibred $T$-set $(T\times Z)/X_{\alpha}$.  Then $[(X,\,\alpha),\, \sigma]=[GGT/D_{\sigma,\, \alpha,\, X}]$ in $B_T(G\times G)$. 
\end{lema}
\begin{proof}
Notice first that $G\times_Z(X,\,\alpha)$ is a transitive $GGT$-set. This is because for $g\in G$ and $(t,\,z)\in T\times Z$, we have
\begin{displaymath}
[g,\,(t,\,z)X_{\alpha}]=[gz,\, (t,\,1)X_{\alpha}]=(gz,\,1,\,t)[1,\, (1,\,1)X_{\alpha}].
\end{displaymath}
Now suppose that $(g_1,\,g_2,\,t)\in GGT$ is in the stabilizer of $[1,\, (1,\,1)X_{\alpha}]$. Then there exists $z\in Z$ such that
\begin{displaymath}
(g_1\sigma(g_2)^{-1},\, (t,\,1)X_{\alpha})=(z,\, (1,\, z^{-1})X_\alpha).
\end{displaymath}
Hence $g_1\sigma(g_2)^{-1}=z$ and $(t,\,1)X_{\alpha}=(1,\,z^{-1})X_{\alpha}$. This implies $(t,\,z)\in X_{\alpha}$ and $(g_1,\, g_2,\,t)\in D_{\sigma,\, \alpha,\, X}$. The other inclusion is clear.
\end{proof}

The previous lemma prompt us to consider, first the ring $kB^Z(T)$, and then the following $k$-algebra, which we denote by $kB^Z(T)\rtimes Out(G)$.

As a $k$-module 
\begin{displaymath}
kB^Z(T)\rtimes Out(G):=\bigoplus_{\sigma\in Out(G)}kB^Z(T).
\end{displaymath}
If the element $u\in kB^Z(T)$ is in the component of this direct sum indexed by $\sigma$, then we will write $u_{\sigma}$.  

To define a product in $kB^Z(T)\rtimes Out(G)$ we introduce the following notation. If $\tau$ is an automorphism of $G$ and $U$ is a $Z$-fibred $T$-set, the $(T\times Z)$-set $\tau(U)$ is the set $U$ with the action of $T\times Z$ given by 
\begin{displaymath}
(t,\, z)\cdot u=(t,\, \tau^{-1}(z))u.
\end{displaymath}
Notice that $\tau(U)$ is also a $Z$-fibred $T$-set, so, this \textit{twist} defines an action of $Out(G)$ on $kB^Z(T)$, which we continue to denote by $\tau(u)$, for $\tau\in Out(G)$ and $u\in kB^Z(T)$. With this, the product in $kB^Z(T)\rtimes Out(G)$ is defined as 
\begin{displaymath}
v_{\tau}\cdot u_{\sigma}:=(v\, \tau(u))_{\tau\sigma}.
\end{displaymath}

By Remark \ref{base}, a $k$-basis of $kB^Z(T)\rtimes Out(G)$ is given by $[X,\,\alpha]_\sigma$, for $X\leqslant T$, a group homomorphism $\alpha:X\rightarrow Z$ and $\sigma\in Out(G)$. With this notation, the identity element of $kB^Z(T)\rtimes Out(G)$ is $[T,\,\uno]_{id}$, where $id:G\rightarrow G$ is the identity element and $\uno :T\rightarrow Z$ is the trivial morphism.

\begin{teo}
\label{elteo}
Let $G$ and $T$ be finite groups. There is an injective morphism of $k$-algebras from $kB^{Z(G)}(T)\rtimes Out(G)$ to $kB_T(G\times G)$. Moreover, by composition with $E_{T,\, G}$, this morphism extends to an injective morphism   $kB^{Z(G)}(T)\rtimes Out(G)\rightarrow\widehat{kB_T}(G)$ and its image is the subalgebra generated by $\mathcal{L}_p$.
\end{teo}
\begin{proof}
We define 
\begin{displaymath}
\Phi: kB^{Z(G)}(T)\rtimes Out(G)\rightarrow kB_T(G\times G)
\end{displaymath}
by sending $[(X,\,\alpha)]_{\sigma}$ to $[(X,\,\alpha),\,\sigma]=[GGT/D_{\sigma,\, \alpha,\, X}]$ in $kB_T(G\times G)$. To see that this is well defined, suppose $[(Y,\, \beta)]_{\tau}=[(X,\,\alpha)]_{\sigma}$. By Notation \ref{delmono}, this means that there exists $t\in T$ such that $(X,\, \alpha)=({ }^tY,\, { }^t\beta)$ and that $\sigma=\tau$ in $Out(G)$. Hence, there exists $(g,\,t)\in G\times T$ such that $X= ^t\!\!Y$, $\sigma=\tau c_g$ and $\alpha=\beta c_{t^{-1}}$. Then, by Lemma \ref{sonigual}, we have $[GGT/D_{\sigma,\, \alpha,\, X}]=[GGT/D_{\tau,\, \beta,\, Y}]$. Going backwards, Lemma \ref{sonigual}  shows too that $\Phi$ is injective. 

To see that $\Phi$ is a ring homomorphism, notice first that if $\tau$ is an automorphism of $G$, then $\tau((X,\,\alpha))=(X,\, \tau\alpha)$. Indeed, if $(t,\,z)\in T\times Z$ is such that $(t,\,\tau^{-1}(z))X_\alpha=X_\alpha$, then, $t\in X$ and $z=\tau\alpha(t)$. Hence, by Remark \ref{base} and Lemma \ref{composicion}, we have that $\Phi$ is a ring homomorphism. Also, $\Phi([T,\, \uno]_{id})=[GGT/D_{id, \uno,\, T}]$. 

The rest of the statement follows from lemmas \ref{Einy} and \ref{lpinls}.
\end{proof}

\begin{coro}
\label{isoprimos}
Suppose  $G$ and $T$ are groups such that $(|G|,\, |T|)=1$. Then   $\mathcal{L}_p=\mathcal{L}_{s'}=\mathcal{G}$ is a basis for $\widehat{kB_T}(G)$ and $\widehat{kB_T}(G)\cong kB^{Z(G)}(T)\rtimes Out(G)$. Furthermore, in this case we have an isomorphism of $k$-modules,
\begin{displaymath}
kB_T(G\times G)\cong (kB^{Z(G)}(T)\rtimes Out(G))\oplus \mathcal{I}_T(G).
\end{displaymath}
\end{coro}
\begin{proof}
We show first that $\mathcal{G}$ is contained in $\mathcal{L}_p$. 

Let $D\leqslant GGT$ be a group such that $p_1(D)=G$, $k_1(D)=\uno$, $p_2(D)=G$ and $k_2(D)=\uno$. Then, 
\begin{displaymath}
D=\{(f(g,\, t),\, g,\, t)\mid (g,\, t)\in A\},
\end{displaymath}
where $A\leqslant G\times T$ satisfies $p_1(A)=G$ and $f:A\rightarrow G$ is a surjective group homomorphism such that $Ker f\cap (k_1(A)\times \uno)=\uno$. Since $(|G|,\, |T|)=1$, we  have that $A=G\times T_0$ for some $T_0\leqslant T$. As in the paragraph before Definition \ref{deflp}, this gives  $D=D_{\sigma,\, \alpha,\, T_0}$ for some $\sigma$ automorphism of $G$ and $\alpha :T_0\rightarrow Z(G)$ a group homomorphism.

With this we have the equality $\mathcal{L}_p=\mathcal{L}_{s'}=\mathcal{G}$. Since $\mathcal{G}$ is a generating set for $\widehat{kB_T}(G)$ and in this case it is linearly independent, then it is a basis for $\widehat{kB_T}(G)$. Also, by Lemma \ref{Einy}, this implies  that $E_{T,\,G}:kB_T(G\times G)\rightarrow \widehat{kB_T}(G)$ is a split epimorphism and $kB_T(G\times G)\cong \widehat{kB_T}(G)\oplus \mathcal{I}_T(G)$.

\end{proof}

\section{(Counter)examples}

Throughout this section $k$ is a field of characteristic $0$ and we consider the three sets we have been associating to $\widehat{kB_T}(G)$. That is, $\mathcal{G}$ (Definition \ref{notzeroimplies}), $\mathcal{L}_{s'}$ (Definition \ref{defls}) and $\mathcal{L}_p$ (Definition \ref{deflp}). In what follows we suppose that $\gamma$ is a basis of $\widehat{kB_T}(G)$, contained in $\mathcal{G}$ and which contains $\mathcal{L}_{s'}$. Such a basis exists, for example, by Theorem 2 of Chapter II, Section 7-1, in \cite{bour}. 

By means of the GAP program below, we computed two examples. The first example  provides groups $T$ and $G$ for which we have the proper inclusions 
\begin{displaymath}
\mathcal{L}_p\subset\mathcal{L}_{s'}\subset  \gamma \subset \mathcal{G}.
\end{displaymath}
It also shows that the morphism $\widehat{\alpha}_{GGT}$ is not injective in general. We include too an example with abelian groups $G$ and $T$, to show that not even with this extra hypothesis, we obtain equality between all of the sets above (except for $\gamma$ and $\mathcal{L}_{s'}$ in this case).

Given a finite group $T$, the program runs through a list of groups, and for each $G$ in this list it computes the following:
\begin{enumerate}
\item[Gen:] The cardinality of $Gen$, a set of representatives of the conjugacy classes of subgroups $D$ of $GGT$,  that satisfy $p_1(D)=G$, $p_2(D)=G$, $k_1(D)=\uno$ and $k_2(D)=\uno$. The image under $E_{T,\, G}$ of their corresponding  elements, $[GGT/D]\in kB_T(G\times G)$, is the set $\mathcal{G}$.
\item[St$'$:] The number of subgroups $D$ in $Gen$ that cannot be written as $D=E\ast F$ for some $E\leqslant GHT$, $F\leqslant HGT$ and $H$ a group with $|H|<|G|$. The image under $E_{T,\, G}$ of their corresponding elements, $[GGT/D]\in kB_T(G\times G)$, is the set $\mathcal{L}_{s'}$ and, by Lemma \ref{Einy}, St$'$ is  the cardinality of $\mathcal{L}_{s'}$. 
\item[Dim:] The dimension of the essential algebra $\widehat{kB_T}(G)$.
\item[Prod:] The sum 
\begin{displaymath}
\sum_{X\in [\mathscr{S}_T]}  |Out (G)||\overline{Hom}(X,\, Z(G))|, 
\end{displaymath}
where $\overline{Hom}(X,\, Z(G))$ is the set of orbits of $Hom(X,\, Z(G))$ under the right action of $N_T(X)$ described in  Section 1.3. By Theorem \ref{elteo}, Prod is the cardinality of $\mathcal{L}_p$.
\end{enumerate}

\subsection{General case}
This example  provides groups $T$ and $G$ for which we have the proper inclusions 
\begin{displaymath}
\mathcal{L}_p\subset\mathcal{L}_{s'}\subset  \gamma \subset \mathcal{G}.
\end{displaymath}
 We also show that the morphism $\widehat{\alpha}_{GGT}$ is not injective in this case.

For $T=Q_8$, the quaternion group, and $G=C_4$ we have
\begin{center}
Gen: 58, $\quad$ St$'$: 46, $\quad$ Dim: 52, $\quad$ Prod: 32.
\end{center}
The proper inclusions $\mathcal{L}_p\subset\mathcal{L}_{s'}\subset  \gamma$ are clear. In particular, there exists an element $v=[[GGT/D]]\in \gamma$ such that $D=E\ast F$ for some $E\leqslant GHT$, $F\leqslant HGT$ and $|H|<|G|$. Hence $0\neq v\in \widehat{kB_T}(G)$ and clearly $\widehat{\alpha}_{GGT}(v)=0$, so $\widehat{\alpha}_{GGT}$ is not injective. 

Notice that Gen being bigger than Dim does not imply that $\gamma$ is properly contained in $\mathcal{G}$, since we do not know, in general, exactly what happens to $Gen$ in $\widehat{kB_T}(G)$. To show this proper inclusion, we consider the group $\Delta_i=\langle(a,\, a,\,i)\rangle\leqslant GGT$, where $a$ is a generator of $G$ and $i$ is an element of order four in $Q_8$. Observe that $p_1(\Delta_i)=G$, $k_1(\Delta_i)=\uno$, $p_2(\Delta_i)=G$ and $k_1(\Delta_i)=\uno$. Hence, $v_i=[[GGT/\Delta_i]]\in \mathcal{G}$. If $v_i=0$, we are done. Otherwise, we consider the groups 
\begin{displaymath}
E=\langle (a,\, 1,\,i)\rangle\leqslant G\uno T\quad\textrm{and}\quad F=\langle (1,\, a,\,i)\rangle\leqslant \uno G T.
\end{displaymath}
We easily have $E\ast F=\Delta_i$. So, the composition
\begin{displaymath}
[G\uno T/E]\times^d_{\uno}[\uno GT/F]=\sum_{t\in[T/\langle i\rangle]}[GGT/E\ast{}^{(1,\, 1,\, t)}F]
\end{displaymath}
is easily seen to be equal to 
\begin{displaymath}
[GGT/\Delta_i]+[GGT/\Delta'_{i}],
\end{displaymath}
where $\Delta'_{i}=\langle(a,\, a^3,\,i)\rangle$. 

As before, $v'_i=[[GGT/\Delta'_i]]\in \mathcal{G}$. Hence, 
we have  $v_i$, $v'_i$ in $\mathcal{G}$ such that $v_i+v'_i=0$ in $\widehat{kB_T}(G)$. So,  $\gamma$ is properly contained in $\mathcal{G}$.

\subsection{Abelian case}
This example provides abelian groups $T$ and $G$ for which we have the proper inclusions 
\begin{displaymath}
\mathcal{L}_p\subset\mathcal{L}_{s'}\subset  \mathcal{G}.
\end{displaymath}
Recall that in this case $\mathcal{L}_{s'}$ is a basis of $\widehat{kB_T}(G)$. 

For $T=C_4$  and $G=C_4$ we have
\begin{center}
Gen: 22, $\quad$ St$'$: 16, $\quad$ Dim: 16, $\quad$ Prod: 14.
\end{center}
The proper inclusion $\mathcal{L}_p\subset\mathcal{L}_{s'}$ is clear. Also, since Gen is bigger than Dim, we have that there are groups $D$ in $Gen$ such that $D=E\ast F$ for some $E\leqslant GHT$, $F\leqslant HGT$ and $|H|<|G|$. In this case, the proof of Theorem \ref{iso2abe} gives us that $[[GGT/D]]=0$ in $\widehat{kB_T}(G)$ and, hence, that the zero vector is in $\mathcal{G}$.

\subsection{GAP program}

\begin{small}
\begin{lstlisting}
# Double shifted Burnside ring
#
# Computing the "star" product. The parameters are:
# l <= ght=G x H x T, m <= hkt=H x K x T, gkt=G x K x T
# p23lintp13m=p_23(L)\cap p_13(M)
# ph=(H x T -> H) pt=(H x T -> T)
#
etoile3:=function(ght,hkt,l,m,gkt,p23lintp13m,ph,pt)
local p1ght,p2ght,p3ght,e1ght,e2ght,e3ght,p1hkt,p2hkt,p3hkt,e1hkt,e2hkt,
e3hkt,p2lintp1m,e1gkt,e2gkt,e3gkt,k1,k2,s,u,v,w,phi,psi,g,h,k;
	p1ght:=Projection(ght,1);
	g:=Image(p1ght);
	p2ght:=Projection(ght,2);
	h:=Image(p2ght);
	p3ght:=Projection(ght,3);
	e1ght:=Embedding(ght,1);
	k1:=GeneratorsOfGroup(Intersection(l,Image(e1ght)));
	e2ght:=Embedding(ght,2);
	e3ght:=Embedding(ght,3);
	p1hkt:=Projection(hkt,1);
	p2hkt:=Projection(hkt,2);
	k:=Image(p2hkt);
	p3hkt:=Projection(hkt,3);
	e1hkt:=Embedding(hkt,1);
	e2hkt:=Embedding(hkt,2);
	e3hkt:=Embedding(hkt,3);
	k2:=GeneratorsOfGroup(Intersection(m,Image(e2hkt)));
	e1gkt:=Embedding(gkt,1);
	e2gkt:=Embedding(gkt,2);
	e3gkt:=Embedding(gkt,3);
	k1:=List(k1,x->Image(e1gkt,Image(p1ght,x)));
	k2:=List(k2,x->Image(e2gkt,Image(p2hkt,x)));
	phi:=GroupHomomorphismByFunction(p23lintp13m,ght,function(x) return Image(e2ght,Image(ph,x))*Image(e3ght,Image(pt,x));end);
	psi:=GroupHomomorphismByFunction(p23lintp13m,hkt,function(x) return Image(e1hkt,Image(ph,x))*Image(e3hkt,Image(pt,x));end);
	s:=GeneratorsOfGroup(p23lintp13m);
	v:=[];
	for u in s do
		w:=Image(e3gkt,Image(pt,u));
		Add(v,Image(e1gkt,First(g,x->Image(e1ght,x)*Image(phi,u) in l))*Image(e2gkt,First(k,y->Image(e2hkt,y)*Image(psi,u) in m))*w);
	od;
	return Subgroup(gkt,Concatenation(k1,k2,v));
end;

corps:=Rationals; # the field of coefficients
t:=QuaternionGroup(8); # the group T
rclt:=List(ConjugacyClassesSubgroups(t),Representative);
st:=StructureDescription(t);
sortie:=Concatenation("essential-double-",st,".res"); # the output file
PrintTo(sortie,"Group T = ",StructureDescription(t)," :\n******************************\n");
listeg:=AllGroups([1..20]); # list of groups G
for g in listeg do
	Print("\n***********************************\n");
	Print("Group G : ",StructureDescription(g),"\n");
	Print("***********************************\n");
	AppendTo(sortie,"\n***********************************\n");
	AppendTo(sortie,"Group G : ",StructureDescription(g),"\n");
	AppendTo(sortie,"***********************************\n");
	hpossible:=Filtered(listeg,x->Size(x)<Size(g)); # groups H through which factorization may occur in the double shifted Burnside ring
	ggt:=DirectProduct(g,g,t);
	p1ggt:=Projection(ggt,1);
	p2ggt:=Projection(ggt,2);
	p3ggt:=Projection(ggt,3);
	e1ggt:=Embedding(ggt,1);
	e2ggt:=Embedding(ggt,2);
	e3ggt:=Embedding(ggt,3);
	k1ggt:=Kernel(p1ggt);
	k2ggt:=Kernel(p2ggt);
	k3ggt:=Kernel(p3ggt);
	k23:=Intersection(k2ggt,k3ggt);
	k13:=Intersection(k1ggt,k3ggt);
	cggt:=ConjugacyClassesSubgroups(ggt);
# we keep only those subgroups of G x G x T with full first and second projection
	cggt:=Filtered(cggt,x->Image(p1ggt,Representative(x))=g and Image(p2ggt,Representative(x))=g);
# then we keep only those subgroups with trivial k1 and k2
	cggt:=Filtered(cggt,x->IsTrivial(Intersection(Representative(x),k23)) and IsTrivial(Intersection(Representative(x),k13)));
# the essential algebra will be a quotient of the vector space with basis 'cggt'
	lcggt:=Length(cggt);
	Print(lcggt," conjugacy classes of subgroups\n");
	AppendTo(sortie,lcggt," conjugacy classes of subgroups\n");
	ng:=Size(g);
	factetoile:=[];
	mat:=[]; # generators of the subspace to factor out in the vector space with basis 'cggt'
# start the test for factorization
	for h in hpossible do
		strh:=StructureDescription(h);
		Print("\n===============\nGroup H ",strh,"\n");
		AppendTo(sortie,"\n===============\nGroup H ",strh,"\n");
		ght:=DirectProduct(g,h,t);
		rcght:=List(ConjugacyClassesSubgroups(ght),Representative);
		hgt:=DirectProduct(h,g,t);
		p1ght:=Projection(ght,1);
		p2ght:=Projection(ght,2);
		p3ght:=Projection(ght,3);
		kk:=Intersection(Kernel(p2ght),Kernel(p3ght));
	# we keep only those subgroups of G x H x T with full projection to G
		rcght:=Filtered(rcght,x->Image(p1ght,x)=g);
	# then we keep only those subgroups with trivial k1
		rcght:=Filtered(rcght,x->IsTrivial(Intersection(x,kk)));
		Print("remaining ",Length(rcght)," conjugacy classes of subgroups\n");
		AppendTo(sortie,"remaining ",Length(rcght)," conjugacy classes of subgroups\n");
		lrcght:=Length(rcght);
		e1hgt:=Embedding(hgt,1);
		e2hgt:=Embedding(hgt,2);
		e3hgt:=Embedding(hgt,3);
		p1hgt:=Projection(hgt,1);
		p2hgt:=Projection(hgt,2);
		p3hgt:=Projection(hgt,3);
	# we build the isomorphism sigma : G x H x T -> H x G x T switching the first two components
		sigma:=GroupHomomorphismByFunction(ght,hgt,function(x) return 
            Image(e1hgt,Image(p2ght,x))*Image(e2hgt,Image(p1ght,x))*
            Image(e3hgt,Image(p3ght,x));end);
	# we transform the previous list of subgroups by sigma
		rchgt:=List(rcght,x->Image(sigma,x));
		ht:=DirectProduct(h,t);
		p1ht:=Projection(ht,1);
		p2ht:=Projection(ht,2);
		e1ht:=Embedding(ht,1);
		e2ht:=Embedding(ht,2);
		p23:=GroupHomomorphismByFunction(ght,ht,function(x) return Image(e1ht,Image(p2ght,x))*Image(e2ht,Image(p3ght,x));end);
		p13:=GroupHomomorphismByFunction(hgt,ht,function(x) return Image(e1ht,Image(p1hgt,x))*Image(e2ht,Image(p3hgt,x));end);
		for kl in [1..Length(rcght)] do
			l:=rcght[kl]; # the subgroup L of G x H x T
			p23l:=Image(p23,l);
			for km in [1..Length(rchgt)] do
				m:=rchgt[km]; # the subgroup M of H x G x T
				Print(kl,"|",km,"/",lrcght,"   \r");
				# compute the composition (G x H x T)/L o (H x G xT)/M in the vector space with basis 'cggt'
				vec:=List(cggt,x->Zero(corps));
				p13m:=Image(p13,m);
				# apply the Mackey formula for the composition
				d:=DoubleCosetRepsAndSizes(ht,p23l,p13m);
				d:=List(d,x->x[1]^-1);
				for z in d do
					hz:=Image(p1ht,z);
					tz:=Image(p2ht,z);
					gam:=Image(e1hgt,hz)*Image(e3hgt,tz);
					p23lintp13m:=Intersection(p23l,p13m^z); 
					et:=etoile3(ght,hgt,l,m^gam,ggt,p23lintp13m,p1ht,p2ht);
					tst:=First([1..Length(cggt)],x->et in cggt[x]);
					if tst<>fail then
						factetoile:=Union(factetoile,[tst]); # the subgroup 'et' can be factored as U * V
					fi;
					if Image(p1ggt,et)=g and Image(p2ggt,et)=g then
						p:=First([1..lcggt],x->et in cggt[x]); # is 'et' in the basis 'cggt'
						if p<>fail then
							vec[p]:=vec[p]+One(corps);
						fi;
					fi;
				od;
				Add(mat,vec); # adds the vector 'vec' to the list 'mat' of generators of the subspace
			od;
		od;
		Print(Length(factetoile)," subgroups can be factored\n");
		AppendTo(sortie,Length(factetoile)," subgroups can be factored\n");
	od;
# computes the dimension of the subspace
	if mat=[] then
		r:=0;
	else
		r:=RankMat(mat);
	fi;
	Print("---------------------------\n");
	Print("\nRemaining ",Length(cggt)-Length(factetoile)," subgroups <> U*V\n");
	AppendTo(sortie,"---------------------------\n");
	AppendTo(sortie,"\nRemaining ",Length(cggt)-Length(factetoile)," subgroups <> U*V\n");
	Print("Essential algebra of dimension ",lcggt," - ",r," = ",lcggt-r,"\n");
	AppendTo(sortie,"Essential algebra of dimension ",lcggt," - ",r," = ",lcggt-r,"\n");
# now compute the sum Sum_T|N(T)\\Hom(T,Z(G))||Out(G)|
	d:=0;
	zg:=Center(g);
	for tt in rclt do
		homs:=AllHomomorphisms(tt,zg);
		nhoms:=Length(homs);
		gtt:=GeneratorsOfGroup(Normalizer(t,tt));
		s:=[];
		for u in gtt do
			Add(s,PermList(List([1..nhoms],x->Position
			(homs,GroupHomomorphismByFunction(tt,zg,function(y) return Image(homs[x],y^u);end)))));
		od;
		s:=Set(s);
		s:=Group(s);
		d:=d+Length(Orbits(s,[1..nhoms]));
	od;
	autg:=AutomorphismGroup(g);
	inng:=InnerAutomorphismsAutomorphismGroup(autg);
	soutg:=Size(autg)/Size(inng); # |Out(G)|
	d:=d*soutg;
	Print("Sum_T|N(T)\\Hom(T,Z(G))||Out(G)| = ",d,"\n");
	AppendTo(sortie,"\nSum_T|N(T)\\Hom(T,Z(G))||Out(G)| = ",d,"\n");
od;
\end{lstlisting}

\end{small}

\section*{Acknowledgments}

The author is very thankful to Serge Bouc for his help with the GAP program. Also, many thanks to the referee for suggesting to consider the monomial Burnside ring in Section 3, thus giving a much nicer statement of Theorem \ref{elteo}.

\bibliographystyle{plain}
\bibliography{esencial}

\begin{thebibliography}{10}

\bibitem{bamon}
Laurence Barker.
\newblock Fibred permutation sets and the idempotents and units of monomial
  {B}urnside rings.
\newblock {\em Journal of Algebra}, 281:535--566, 2004.

\bibitem{bolco}
Robert Boltje and Olcay Co\c{s}kun.
\newblock Fibered biset functors.
\newblock {\em Advances in mathematics}, 339:540--598, 2018.

\bibitem{boltje-danz}
Robert Boltje and Susanne Danz.
\newblock A ghost algebra of the double {B}urnside algebra in characteristic
  zero.
\newblock {\em Journal of Pure and Applied Algebra}, 217:608--635, 2013.

\bibitem{ensemble}
Serge Bouc.
\newblock Foncteurs d'ensembles munis d'une double action.
\newblock {\em Journal of Algebra}, 183:664--736, 1996.

\bibitem{biset}
Serge Bouc.
\newblock {\em Biset functors for finite groups}.
\newblock Springer, Berlin, 2010.

\bibitem{b-deniz}
Serge Bouc and Deniz~Y\i lmaz.
\newblock Diagonal $p$-permutation functors.
\newblock {\em J. Algebra}, 556:1036--1056, 2020.

\bibitem{centros}
Serge Bouc and Nadia Romero.
\newblock The center of a {G}reen biset functors.
\newblock {\em Pacific Journal of Mathematics}, 303:459--490, 2019.

\bibitem{corr1}
Serge Bouc and Jacques Th\'evenaz.
\newblock The algebra of essential relations on a finite set.
\newblock {\em J. reine angew. Math.}, 712:225--250, 2016.

\bibitem{corr2}
Serge Bouc and Jacques Th\'{e}venaz.
\newblock Correspondence functors and finiteness conditions.
\newblock {\em J. Algebra}, 495:150--198, 2018.

\bibitem{bour}
Nicolas Bourbaki.
\newblock {\em Algebra I. Chapters 1-3}.
\newblock Springer-Verlag, Berlin, 1989.

\bibitem{dmon}
Andreas Dress.
\newblock The ring of monomial representations {I}. {S}tructure theory.
\newblock {\em Journal of Algebra}, 18:137--157, 1971.

\bibitem{maxime}
Maxime Ducellier.
\newblock A study of a simple $p$-permutation functor.
\newblock {\em Journal of Algebra}, 447:367--382, 2016.

\bibitem{GAP4}
The GAP~Group.
\newblock {\em {GAP -- Groups, Algorithms, and Programming, Version 4.11.1}},
  2021.
\newblock \verb+(http://www.gap-system.org)+.

\bibitem{lachica}
Nadia Romero.
\newblock Simple modules over {G}reen biset functors.
\newblock {\em Journal of Algebra}, 367:203--221, 2012.

\bibitem{lachica2}
Nadia Romero.
\newblock On fibred biset functors with fibres of order prime and four.
\newblock {\em Journal of Algebra}, 387:185--194, 2013.

\end{thebibliography}

\end{document}